\documentclass[12pt]{article}
\usepackage{amsmath}
\usepackage{amsfonts}
\usepackage{pb-diagram}
\usepackage{amsthm} 
\usepackage{amscd}
\usepackage{amssymb}
\usepackage{mathdots}
\newcounter{numb}

\theoremstyle{plain} 
\newtheorem{theorem}{Theorem}[section]

\newtheorem{lemma}[theorem]{Lemma}

\newtheorem{remark}[theorem]{Remark}
\numberwithin{equation}{section}

\newcommand{\al}{\alpha}

\newcommand{\br}{\mathbf {R}}

\newcommand{\x}{\mathcal{X}}

\newcommand{\w}{\wedge}
\newcommand{\ld}{\ldots}
\newcommand{\p} {\partial}
\newcommand{\g} {\mathfrak {g}}

\newcommand{\h}{\mathfrak {h}}
\newcommand{\I}{\mathfrak {I}}

\newcommand{\so}{\mathfrak {so}}
\begin{document}
\title{Leibniz Homology of the Affine Indefinite Orthogonal Lie Algebra}
\author{Guy Roger Biyogmam}
\date{}
\maketitle{.\\Department of Mathematics,\\Southwestern Oklahoma State University,\\Weatherford, OK 73096, USA\\ \textit{Email:}\textit{guy.biyogmam@swosu.edu}}
\begin{abstract}
In this paper, we compute the Leibniz (co)homology of  the  affine indefinite orthogonal Lie algebra. This calculation  generalizes a result  \cite[corollary 4.5]{JL}  obtained by Jerry Lodder.  We  construct several  indefinite orthogonal   invariants  in terms of balanced tensors and provide the Leibniz homology in terms of these invariants.
\end{abstract}
\textbf{Mathematics Subject Classifications(2000):} 17B56, 17A32, 17B99.\\
\textbf{Key Words}: Leibniz algebras, Leibniz homology, Indefinite orthogonal Lie algebra. 
\section{Introduction}
The indefinite orthogonal group is among  the most important groups that are broadly used in physics. Particular cases of interest are made explicit in a variety of papers and their importance, especially in quantum field theory \cite{WS}, higher energy physics \cite{LM} and cosmology \cite{KB}, make them remarkable. Just to mention a few, we point out the conformal group $SO(4,2)$ of the Minkowski space \cite{KHA} also known as the dynamical group of the $3$-dimensional non-realistic quantum mechanic Kepler problem, the deSitter group $SO(4,1)$ and the anti-deSitter group $SO(3,2),$ the Lorentz group $SO(3,1)$ \cite{W}, the split orthogonal group $SO(n,n)$ which is shown to be a Chevalley group \cite{VA}, and the quasi-split orthogonal group $SO(n,n+1)$  which is shown to be a Steinberg group \cite{CPS}. Also, the representation theory \cite{W, RJ, BZ, T} of these groups have been seriously investigated. Recall that the indefinite orthogonal group \cite{HS} $O(p,q),$ $p,q\in \mathbb{N}$ with  $p+q=n,$ consists of the matrix $M\in GL(n,\br)$ satisfying $M^TI_{p,q}M=I_{p,q}$ where  

$$I_{p,q}=\begin{bmatrix} I_p & 0 \\ 0 & -I_q \end{bmatrix}$$

with $I_k$ denoting the $k\times k$ identity matrix. Note that $O(p,q)\cong O(q,p)$ via the isomorphisms $\phi; O(p,q)\rightarrow O(q,p)$ with $\phi( M)=\sigma M\sigma$ where $\sigma$ is the skew diagonal matrix $$\sigma=\begin{bmatrix} 0 & \cdots & 1 \\ \vdots &  \iddots & \vdots \\ 1  &  \cdots & 0 \end{bmatrix}$$

With the Poincare group $\br^{3,1}\rtimes SO(3,1)$ as a model of departure, we focus  for $n=p+q,$ $p,q\geq 1$ integers, on the affine indefinite group $\br^{p,q}\rtimes SO(p,q)$  where $\br^{p,q}\cong\br^ p\oplus\br^q$ is the real vector space equipped with the quadratic form $$(u,v)=\sum_{i=1}^px_iy_i-\sum _{i=p+1}^{p+q}x_iy_i$$ with     $u=(x_1,\ldots,x_{p+q})$ and $v=(y_1,\ldots,y_{p+q})$. Note that $SO(p,q)$ acts on $\br^{ p,q}$ via left multiplication or standard representation.

Denote by $\h_n,$ the Lie algebra of the affine indefinite orthogonal group. We calculate in this paper several indefinite orthogonal invariants, and $\h_n$-invariants which are detected by Leibniz homology via Lodder's structure theorem \cite{JL}. The tools used to compute these invariants are inspired by the author's previous work for the definite case. The main result of the paper is  the  isomorphism of graded vector spaces $$HL_*(\h_n)\cong (\br\oplus\left\langle \tilde{\alpha}_{p,q}\right\rangle)\otimes T^*(\tilde{\gamma}_{p,q}),$$

 where $\left\langle \tilde{\alpha}_{p,q}\right\rangle$  denotes a 1-dimensional vector space in degree $n$
on $$\tilde{\alpha}_n=\sum_{\sigma\in S_n}\mbox{sgn} (\sigma)\frac{\p}{\p x^{\sigma(1)}}\otimes\frac{\p}{\p x^{\sigma(2)}}\otimes\frac{\p}{\p x^{\sigma(3)}}\otimes\ldots\otimes\frac{\p}{\p x^{\sigma(n)}}$$  and $T^*(\tilde{\gamma}_{p,q})$ denotes the tensor algebra on the $(n-1)$-degree
generator  $\tilde{\gamma}_{p,q}=\bar{\gamma}_{p,q}-\bar{\gamma}'_{p,q}$ with 

 $$
\begin{aligned}\bar{\gamma}_{p,q}=&\frac{1}{n!}\big(~\sum_{\substack{1\leq i<j\leq p, \\ \sigma\in S_{n-2}}}(-1)^{i+j+1}sgn(\sigma) X_{ij}\otimes\frac{\p}{\p x^{\sigma(1)}}\otimes\ldots\widehat{\frac{\p}{\p x^{\sigma(i)}}}\ldots\widehat{\frac{\p}{\p x^{\sigma(j)}}}\ldots\otimes\frac{\p}{\p x^{\sigma(n)}}\\ &- \sum_{\substack{p+1\leq i<j\leq n\\ \sigma\in S_{n-2}}} (-1)^{i+j+p+1}sgn(\sigma) X_{ij}\otimes\frac{\p}{\p x^{\sigma(1)}}\otimes\ldots\widehat{\frac{\p}{\p x^{\sigma(i)}}}\ldots\widehat{\frac{\p}{\p x^{\sigma(j)}}}\ldots\otimes\frac{\p}{\p x^{\sigma(n)}}\\ &+ \sum_{\substack{1\leq i\leq p\\ p+1\leq j\leq n\\ \sigma\in S_{n-2}}} (-1)^{i+j}sgn(\sigma)Y_{ij}\otimes\frac{\p}{\p x^{\sigma(1)}}\otimes\ldots\widehat{\frac{\p}{\p x^{\sigma(i)}}}\ldots\widehat{\frac{\p}{\p x^{\sigma(j)}}}\ldots\otimes\frac{\p}{\p x^{\sigma(n)}}~\big)
\end{aligned}
$$   and

 $$
\begin{aligned}\bar{\gamma}'_{p,q}=&\frac{1}{n!}\big(~\sum_{\substack{1\leq i<j\leq p, \\ \sigma\in S_{n-2}}}(-1)^{i+j+1}sgn(\sigma)\frac{\p}{\p x^{\sigma(1)}}\otimes\ldots\widehat{\frac{\p}{\p x^{\sigma(i)}}}\ldots\widehat{\frac{\p}{\p x^{\sigma(j)}}}\ldots\otimes\frac{\p}{\p x^{\sigma(n)}}\otimes X_{ij}\\ &- \sum_{\substack{p+1\leq i<j\leq n\\ \sigma\in S_{n-2}}}(-1)^{i+j+p+1}sgn(\sigma) \frac{\p}{\p x^{\sigma(1)}}\otimes\ldots\widehat{\frac{\p}{\p x^{\sigma(i)}}}\ldots\widehat{\frac{\p}{\p x^{\sigma(j)}}}\ldots\otimes\frac{\p}{\p x^{\sigma(n)}}\otimes X_{ij}\\ &+ \sum_{\substack{1\leq i\leq p\\ p+1\leq j\leq n\\ \sigma\in S_{n-2}}} (-1)^{i+j}sgn(\sigma)\frac{\p}{\p x^{\sigma(1)}}\otimes\ldots\widehat{\frac{\p}{\p x^{\sigma(i)}}}\ldots\widehat{\frac{\p}{\p x^{\sigma(j)}}}\ldots\otimes\frac{\p}{\p x^{\sigma(n)}}\otimes Y_{ij}~\big)
\end{aligned}
$$ 

This result generalizes Lodder's result obtained on the Poincar\'e group. We show in section 4 that $\tilde{\alpha}_{p,q}$ and $\tilde{\gamma}_{p,q}$ are  $\h_n$-invariant.  
 Dually for cohomology,  there is an isomorphism of dual Leibniz algebras $$HL^*(\h_n)\cong (\br\oplus\left\langle \tilde{\alpha}^d_{p,q}\right\rangle)\otimes T^*(\tilde{\gamma}^d_{p,q}),$$ where 
$\tilde{\alpha}^d_{p,q}$ and $\tilde{\gamma}^d_{p,q}$ are respectively  dual to $\tilde{\alpha}_{p,q}$ and $\tilde{\gamma}_{p,q}$ with respect to the basis of $\h_n.$

\section{The Affine Indefinite Orthogonal Lie Algebra}

Let $\g$ be the Lie algebra of the maximal compact subgroup $O(p)\times O(q)$ of $O(p,q).$ Then $\g\cong\so(p)\oplus\so(q)$ as Lie algebras. So there is a vector space isomorphism $$\h_n\cong \I_n\oplus\g\oplus\l$$ where $\I_n$ denotes the abelian Lie algebra of $\br^{p,q}$ and $\l$ is the orthogonal complement of $\g$ with respect to the killing form of $\so(p,q).$ Roughly speaking, $\h_n$ may be described as the Lie algebra of a special relativity consisting  of translations on space time, two non-abelian Lie algebras of $p$ and $q$-dimensional rotations, and boosts.

Assume that $\br^n$ is given the coordinates $(x_1,x_2,...,x_n),$ and let $\frac{\partial}{\partial x^i,}$ be the unit vector fields parallel to the $x_i$ axes respectively. It is easy to show that the Lie algebra $\so(p,q)$ endowed with the bracket of vector fields is generated by the  vector fields : $$X_{ij}:=-x_i\frac{\partial}{\partial x^j}+x_j\frac{\partial}{\partial x^i},~~~1\leq i< j\leq p,~p+1\leq i< j\leq n$$ and $$Y_{ij}:=x_i\frac{\partial}{\partial x^j}+x_j\frac{\partial}{\partial x^i},~~~1\leq i\leq p,~p+1\leq  j\leq n$$ where the $X_{ij}$ constitute a basis of $\g$ and the $Y_{ij}$ constitute a basis of $\l.$ Also, the Lie algebra  $\I_n$ has   vector space basis: $$\frac{\partial}{\partial x^i},~~1\leq i\leq n.$$ There is a short exact sequence of Lie algebras \cite[p.203]{KA} 
$$0\longrightarrow \I_n\stackrel{i}{\longrightarrow} \h_n\stackrel{\pi}{\longrightarrow} \so(p,q)\longrightarrow 0$$ where $i$ is the inclusion map and $\pi$ is the projection $$\h_n\longrightarrow (\h_n/\I_n)\cong \so(p,q).$$   The bracket on $\h_n\cong\I_n\oplus\so(p,q)$ can be  defined by 
$$[(m_1, x_1), (m_2, x_2)] = ([m_1, x_2] + [x_1, m_2],~[x_1, x_2])~~~~~~~\cite{LP}.$$ 
Also, the Lie algebra $\so(p,q)$  acts on $\I_n$ and on $\h_n$ via the bracket of vector fields. This action is extended to $\I^{\wedge k}_n$ by $$[\al_1\w\al_2\w\ldots\w\al_k,~X]=\sum^k_{i=1}\al_1\w\al_2\w\ldots\w [\al_i,~X]\w\ldots\w\al_k$$ for $\al_i\in\I_n,~X\in\so(p,q),$ and the action of $\so(p,q)$ on $\h_n\otimes\I_n^{\w k}$ is given by  $$[h\otimes\al_1\w\al_2\w\ldots\w\al_k,~X]=[h,~X]\otimes\al_1\w\ldots\w\al_k$$$$+\sum^k_{i=1}h\otimes\al_1\w\al_2\w\ldots\w [\al_i,~X]\w\ldots\w\al_k$$ for $h\in\h_n.$ 

\section{Some Invariants for the Indefinite Orthogonal Lie Algebras.}

In this section, we provide several modules of invariants  under the action of $\so(p,q).$   
Recall that for any  $\so(p,q)$-module $M,$ the submodule  $M^{\so(p,q)}$ of $\so(p,q)-$invariants is defined by $$M^{\so(p,q)}=\left\{ m\in M~|~ [m,~g]=0 ~ \mbox{for all}~g\in \so(p,q)\right\}.$$ 

\begin{lemma}\label{alpha}
 $$[\w^*(\I_n)]^{\so(p,q)}=\br\oplus \left\langle\alpha\right\rangle ~~~~~~~~~~ \mbox{for}~~~ n\geq 4$$ where $$\alpha= \frac{\p}{\p x^1}\w\frac{\p}{\p x^2}\w\ld\w\frac{\p}{\p x^n}, ~~ \mbox{ the volume element.}$$
\end{lemma}
\begin{proof}
  Indeed, that $[\I_n^{\w 0}]^{\so(p,q)}=\br$ is clear. Also it is easy to show that   that $$[\I_n^{\w n}]^{\so(p,q)}=\left\langle \frac{\p}{\p x^1}\w\frac{\p}{\p x^2}\w\ld\w\frac{\p}{\p x^n}\right\rangle.$$ The proof that $[\I_n^{\w 1}]^{\so(p,q)}=\{0\}$ is identical to the definite case \cite{BGR}.  
The standard representation $\I_n$ splits as $\I_p\oplus\I_q$ where $\I_p=\left\langle \frac{\p}{\p x^1},\ld\frac{\p}{\p x^p}\right\rangle$ and $\I_q=\left\langle \frac{\p}{\p x^{p+1}},\ld\frac{\p}{\p x^n}\right\rangle.$ Now we have by  \cite[Lemma 4.1]{BGR} that $[\I_p^{\w k}]^{\so(p,q)}\subseteq [\I_p^{\w k}]^{\so(p)}=0$  for $k\neq 0,1,p$ and   $[\I_q^{\w k}]^{\so(p,q)}\subseteq [\I_q^{\w k}]^{\so(q)}=0$ for $k\neq 0,1,q.$  It follows that   $[\I_n^{\w k}]^{\so(p,q)}=0$ for all $k\neq 0,1,p,q.$ Now since $$[\frac{\p}{\p x^1}\w\frac{\p}{\p x^2}\w\ld\w\frac{\p}{\p x^p},~Y_{1n}]=\frac{\p}{\p x^n}\w\frac{\p}{\p x^2}\w\ld\w\frac{\p}{\p x^p}\neq 0$$ and $$[\frac{\p}{\p x^{p+1}}\w\frac{\p}{\p x^{p+2}}\w\ld\w\frac{\p}{\p x^n},~Y_{1n}]=\frac{\p}{\p x^{p+1}}\w\frac{\p}{\p x^{p+2}}\w\ld\w\frac{\p}{\p x^1}\neq 0,$$ it follows that  $[\I_n^{\w p}]^{\so(p,q)}=0$ and  $[\I_n^{\w q}]^{\so(p,q)}=0.$

\end{proof}

\begin{lemma}\label{dimi}
$$dim[\I_n\otimes\I_n^{\w k}]^{\so(p,q)}=\begin{cases} 1,& \mbox{if} ~~k=1,~n-1 \mbox{}\\ 0,& \mbox{else} \mbox{}\end{cases}$$
\end{lemma}

\begin{proof}
Since $[\I_n\otimes\I_n^{\w k}]^{\so(p,q)}$ and $[\I_n\otimes\I_n^{\w k}]^{\so(n)}$  have the same complexification, it follows that they have the same dimension. We  conclude by \cite[lemma 4.2]{BGR}.

\end{proof}

\begin{lemma}\label{delta}
$$[\I_n\otimes\I_n]^{\so(p,q)}=\left\langle \delta_{p,q}\right\rangle ~~~~\mbox{for}~~  n\geq 4$$ where $$\delta_{p,q}=\sum_{i=1}^{p}\frac{\p}{\p x^i}\otimes\frac{\p}{\p x^i}-\sum_{i=p+1}^{n}\frac{\p}{\p x^i}\otimes\frac{\p}{\p x^i}.$$
\end{lemma}

\begin{proof} By lemma \ref{dimi}, there is only one generator.
Now set $\delta_p=\sum_{i=1}^{p}\frac{\p}{\p x^i}\otimes\frac{\p}{\p x^i}$ and $\delta_q=\sum_{i=p+1}^{n}\frac{\p}{\p x^i}\otimes\frac{\p}{\p x^i}.$ Then by \cite[lemma 4.4]{BGR}, $$[\delta_{p,q},~X_{ij}]=[\delta_p,~X_{ij}]-[\delta_q,~X_{ij}]=0~~~\mbox{for all}~1\leq i<j\leq p,~p+1\leq i<j\leq n.$$ Also it is not hard to check that for $1\leq i\leq p$ and $p+1\leq j\leq n,$ $$[\delta_{p,q},~Y_{ij}]=[\delta_p,~Y_{ij}]-[\delta_q,~Y_{ij}]=(\frac{\p}{\p x^j}\otimes\frac{\p}{\p x^i}+\frac{\p}{\p x^i}\otimes\frac{\p}{\p x^j})-(\frac{\p}{\p x^i}\otimes\frac{\p}{\p x^j}+\frac{\p}{\p x^j}\otimes\frac{\p}{\p x^i})= 0.$$ Hence $[\delta_{p,q}]^{\so(p,q)}=0.$
\end{proof}

\begin{lemma}\label{beta}
$$[\I_n\otimes\w^{n-1}(\I_n)]^{\so(p,q)}=\left\langle \beta_{p,q}\right\rangle~~~~\mbox{for}~~  n\geq 4$$ where 
$$
\begin{aligned}
\beta_{p,q}=&\sum^p_{m=1}(-1)^{m-1}\frac{\p}{\p x^m}\otimes\frac{\p}{\p x^1}\w\frac{\p}{\p x^2}\w\ldots\widehat{\frac{\p}{\p x^m}}\ldots\w\frac{\p}{\p x^p}\w\ldots\w\frac{\p}{\p x^n}\\&-\sum^n_{m=p+1}(-1)^{m-1}\frac{\p}{\p x^m}\otimes\frac{\p}{\p x^1}\w\frac{\p}{\p x^2}\w\ldots\w\frac{\p}{\p x^p}\w\ldots\widehat{\frac{\p}{\p x^m}}\ldots\w\frac{\p}{\p x^n}
\end{aligned}$$
\end{lemma}

\begin{proof}  By lemma \ref{dimi}, there is only one generator.
Set $$\beta_p=\sum^p_{m=1}(-1)^{m-1}\frac{\p}{\p x^m}\otimes\frac{\p}{\p x^1}\w\frac{\p}{\p x^2}\w\ldots\widehat{\frac{\p}{\p x^m}}\ldots\w\frac{\p}{\p x^p}\w\ldots\w\frac{\p}{\p x^n}$$ and $$\beta_q=\sum^n_{m=p+1}(-1)^{m-1}\frac{\p}{\p x^m}\otimes\frac{\p}{\p x^1}\w\frac{\p}{\p x^2}\w\ldots\w\frac{\p}{\p x^p}\w\ldots\widehat{\frac{\p}{\p x^m}}\ldots\w\frac{\p}{\p x^n}.$$
 Then by \cite[Lemma 4.5]{BGR}, $$[\beta_{p,q},~X_{ij}]=[\beta_p,~X_{ij}]-[\beta_q,~X_{ij}]=0~~~\mbox{for all}~1\leq i<j\leq p,~p+1\leq i<j\leq n.$$ Also it is not hard to check that for $1\leq i\leq p$ and $p+1\leq j\leq n,$ 
$$
\begin{aligned}
~~~~~~~~~[\beta_{p,q},~Y_{ij}]&=[\beta_p,~Y_{ij}]-[\beta_q,~Y_{ij}]\\&=[(-1)^{i-1}\frac{\p}{\p x^i}\otimes\frac{\p}{\p x^1}\w\ldots\widehat{\frac{\p}{\p x^i}}\ldots\w\frac{\p}{\p x^p}\w\ldots\w\frac{\p}{\p x^n},~Y_{ij}]\\&~~~~~~~~-[(-1)^{j-1}\frac{\p}{\p x^j}\otimes\frac{\p}{\p x^1}\w\ldots\w\frac{\p}{\p x^p}\w\ldots\widehat{\frac{\p}{\p x^j}}\ldots\w\frac{\p}{\p x^n},~Y_{ij}]\\&=0
\end{aligned}$$ by skew symmetry of the wedge product.  Hence $[\beta_{p,q}]^{\so(p,q)}=0.$
\end{proof}

\begin{lemma}\label{dimo}
$$dim[\so(p,q)\otimes\I_n^{\w k}]^{\so(p,q)}=\begin{cases} 1,& \mbox{if} ~~k=2,~n-2 \mbox{}\\0,& \mbox{else} \mbox{}\end{cases}$$

\end{lemma}
\begin{proof}
 $[\so(p,q)\otimes\I_n^{\w k}]^{\so(p,q)}$ and $[\so(n)\otimes\I_n^{\w k}]^{\so(n)}$ have the same complexification, and thus  the same dimension. The result follows by  \cite[lemma 4.2]{BGR}.

\end{proof}
\begin{lemma}
There is a vector space isomorphism

$$[\so(p,q)\otimes\I_n^{\w 2}]^{\so(p,q)}=\left\langle \rho_{p,q} \right\rangle~~~~\mbox{for}~~  n\geq 4$$ where 
$$
\rho_{p,q}=\sum_{1\leq i<j\leq p}X_{ij}\otimes\frac{\p}{\p x^i}\w\frac{\p}{\p x^j}-\sum_{p+1\leq i<j\leq n}X_{ij}\otimes\frac{\p}{\p x^i}\w\frac{\p}{\p x^j}+\sum_{\substack{1\leq i\leq p\\ p+1\leq j\leq n}} Y_{ij}\otimes\frac{\p}{\p x^i}\w\frac{\p}{\p x^j}.
$$

\end{lemma}

\begin{proof} By lemma \ref{dimo}, there is only one generator.
 It is easy to check the result for  $n=4$ by direct calculation on $\so(3,1)$ and $\so(2,2).$  Suppose the result true for  $\so(p,q-1).$ Note that $[\so(p,q)]^{\so(p,q)}=\{0\}.$ Now let $B_{p,q-1}$ be the vector space basis of $\so(p,q-1).$  It is clear that the vector space basis of $\so(p,q)$ is  $$B_{p,q}=B_{p,q-1}\bigcup\{X_{in},~i=p+1,\ldots,n-1\}\bigcup\{Y_{in},~i=1,\ldots,p\}.$$  Set $S=\{x^1,x^2,\ldots,x^n\}$ and $S'=\{x^1,x^2,\ldots,x^{n-1}\}.$ A vector space basis of $$(\so(p,q)\otimes\I_n^{\w 2})/(\so(p,q-1)\otimes\I_{n-1}^{\w 2})$$ is given by the families of elements: \\$1)~~e\otimes\frac{\p}{\p x^n}\w\frac{\p}{\p z},~~e\in B_{p,q-1},~z\in S'$\\ $2)~~X_{in}\otimes\frac{\p}{\p z^1}\w\frac{\p}{\p z^2},~~z^j\in S,~p+1\leq i\leq n-1$ \\  $3)~~Y_{in}\otimes\frac{\p}{\p z^1}\w\frac{\p}{\p z^2},~~z^j\in S,~1\leq i\leq p.$\\ 
 Now let  $\omega\in [\so(p,q)\otimes\I_n^{\w 2}]^{\so(p,q)}$ with $\omega=u+v$ where $$u\in (\so(p,q-1)\otimes\I_{n-1}^{\w 2}),~~ v\in (\so(p,q)\otimes\I_{n}^{\w 2})/(\so(p,q-1)\otimes\I_{n-1}^{\w 2}).$$ 
Write $v=S_1+S_2+S_3$ with $$S_1=\sum_{z\in S'}c_{1,*}~e\otimes\frac{\p}{\p x^n}\w\frac{\p}{\p z},$$  $$S_2=\sum_{\substack{z^1, z^{2}\in S'\\ p+1\leq i\leq n-1}}c_{2,*}~(X_{in})\otimes\frac{\p}{\p z^1}\w\frac{\p}{\p z^2}+\sum_{\substack{z\in S'\\ p+1\leq i\leq n-1}}c^i_{4,*}~(X_{in})\otimes\frac{\p}{\p z}\w\frac{\p}{\p x^n}$$ and $$S_3=\sum_{\substack{z^1, z^{2}\in S'\\ 1\leq i\leq p}}c_{3,*}~(Y_{in})\otimes\frac{\p}{\p z^1}\w\frac{\p}{\p z^2}+\sum_{\substack{z\in S'\\ 1\leq i\leq p}}c^i_{5,*}~(Y_{in})\otimes\frac{\p}{\p z}\w\frac{\p}{\p x^n}.$$
 For all $\x\in\so(p,q-1)\subseteq\so(p,q),$ as a Lie subalgebra, we have $$0=[\x,~\omega]=[\x,~u]+[\x,~v].$$ This implies that  $[\x,~u]$ and $[\x,~v]$ are zero; otherwise both are non-zero and thus not linearly independent; a contradiction since  $[\x,~v]$ contains at least one term involving the vector field $\frac{\p}{\p x^n}$ in its expression whereas $[\x,~u]$ doesn't. So $u\in [\so(p,q-1)\otimes\I_{n-1}^{\w 2}]^{\so(p,q-1)}$ i.e.  $~u=c\rho_{p,q-1}$ for some constant $c\in\br.$  Since $0=[\x,~v]=[\x,~S_1+~S_2+~S_3]=[\x,~S_1]+[\x,~S_2]+[\x,~S_3],$

we must have $[\x,~S_1]=[\x,~S_2]=[\x,~S_3]=0$ by linear independence.  So $$0=[\x,~S_1]=-\sum_{z\in S'}c_{1,*}~[\x,e\otimes\frac{\p}{\p z}]\w \frac{\p}{\p x^n},$$ thus $$\sum_{z\in S'}c_{1,*}~[\x,e\otimes\frac{\p}{\p z}]=0$$ and thus   $~\sum_{z\in S'}c_{1,*}~e\otimes\frac{\p}{\p z}\in[\so(p,q-1)\otimes\I_{n-1}]^{\so(p,q-1)}=\{0\}.~$ It follows that $S_1=0.$
 
In particular for $\x=X_{p+1n}:=-x_{p+1}\frac{\p}{\p x^n}+x_n\frac{\p}{\p x^{p+1}}\in\so(p,q);$ we have 
$$
\begin{aligned}
0=[\x,~\omega]=&[\x,~u]+[\x,~S_2]+[\x,~S_3]=\\&\sum_{p+2\leq j\leq n-1}c~X_{p+1 j}\otimes\frac{\p}{\p x^j}\w\frac{\p}{\p x^n}-\sum_{p+2\leq j\leq n-1}c~X_{nj}\otimes\frac{\p}{\p x^{p+1}}\w\frac{\p}{\p x^j}\\&+\sum_{1\leq i\leq p}c~Y_{in}\otimes\frac{\p}{\p x^i}\w\frac{\p}{\p x^{p+1}}+\sum_{1\leq i\leq p}c~Y_{i p+1}\otimes\frac{\p}{\p x^{i}}\w\frac{\p}{\p x^n}\\&+\sum_{\substack{z^1, z^{2}\in S\\ p+2\leq i\leq n-1}}c_{2,*}~X_{p+1 i}\otimes\frac{\p}{\p z^1}\w\frac{\p}{\p z^2}+\sum_{\substack{z^1, z^{2}\in S\\ p+1\leq i\leq n-1}}c_{2,*}~X_{in}\otimes[X_{p+1 n},~\frac{\p}{\p z^1}\w\frac{\p}{\p z^2}]\\&-\sum_{\substack{z^1, z^{2}\in S\\ 1\leq i\leq p}}c_{3,*}~Y_{p+1 i}\otimes\frac{\p}{\p z^1}\w\frac{\p}{\p z^2}+\sum_{\substack{z^1, z^{2}\in S\\ 1\leq i\leq p}}c_{3,*}~Y_{in}\otimes[X_{p+1 n},~\frac{\p}{\p z^1}\w\frac{\p}{\p z^2}]\\&+\sum_{\substack{z\in S'\\ p+2\leq i\leq n-1}}c^i_{4,*}~X_{p+1 i}\otimes\frac{\p}{\p z}\w\frac{\p}{\p x^n}+\sum_{\substack{z\in S'\\ p+1\leq i\leq n-1}}c^i_{4,*}~X_{in}\otimes[X_{p+1 n},~\frac{\p}{\p z}\w\frac{\p}{\p x^n}]\\&-\sum_{\substack{z\in S'\\ 1\leq i\leq p}}c^i_{5,*}~Y_{i p+1}\otimes\frac{\p}{\p z}\w\frac{\p}{\p x^n}+\sum_{\substack{z\in S'\\ 1\leq i\leq p}}c^i_{5,*}~Y_{in}\otimes[X_{p+1 n},~\frac{\p}{\p z}\w\frac{\p}{\p x^n}]. 
\end{aligned}
$$
Clearly, all the basis vectors except $X_{p+1 j}\otimes\frac{\p}{\p x^j}\w\frac{\p}{\p x^n},$ $X_{nj}\otimes\frac{\p}{\p x^{p+1}}\w\frac{\p}{\p x^j},$ $Y_{in}\otimes\frac{\p}{\p x^i}\w\frac{\p}{\p x^{p+1}},$ and $Y_{i p+1}\otimes\frac{\p}{\p x^{i}}\w\frac{\p}{\p x^n}$ with $p+2\leq j\leq n-1,~1\leq i\leq p$ appear only once in the summation above. So all the coefficients $c_{2,*},~c_{3,*},~c^i_{4,*},~c^i_{5,*}$  are zero except  $c^i_{4,i}$  and $c^j_{5,j}$with $i\neq p+1$ which satisfy $c^i_{4,i}+c=0$ and $c^j_{5,j}-c=0.$   So $$\omega=c\rho_{p,q-1}-c\sum_{ p+1\leq i\leq n-1} X_{in}\otimes\frac{\p}{\p x^i}\w\frac{\p}{\p x^n}+c\sum_{1\leq i\leq p} X_{in}\otimes\frac{\p}{\p x^i}\w\frac{\p}{\p x^n}+c^{p+1}_{4,p+1 }X_{p+1 n}\otimes\frac{\p}{\p x^{p+1}}\w\frac{\p}{\p x^n}.$$ To finish, we apply  $\x=X_{p+2 n}:=-x_{p+2}\frac{\p}{\p x^n}+x_n\frac{\p}{\p x^{p+2}}$ to the condition $[\x,\omega]=0$ to have $c^{p+1}_{4,p+1}=-c.$ Hence $\omega=c\rho_{p,q}.$


\end{proof}

\begin{lemma}

$$[\so(p,q)\otimes\I_n^{\w n-2}]^{\so(p,q)}=\left\langle \gamma_{p,q} \right\rangle~~~~\mbox{for}~~  n\geq 4$$ where 
$$
\begin{aligned}
\gamma_{p,q}=&\sum_{1\leq i<j\leq p} (-1)^{i+j+1}X_{ij}\otimes\frac{\p}{\p x^1}\w\ldots\widehat{\frac{\p}{\p x^i}}\ldots\widehat{\frac{\p}{\p x^j}}\ldots\w\frac{\p}{\p x^n}\\&-\sum_{p+1\leq i<j\leq n}(-1)^{i+j+p+1}X_{ij}\otimes\frac{\p}{\p x^1}\w\ldots\widehat{\frac{\p}{\p x^i}}\ldots\widehat{\frac{\p}{\p x^j}}\ldots\w\frac{\p}{\p x^n}\\&+\sum_{\substack{1\leq i\leq p\\ p+1\leq j\leq n}}(-1)^{i+j} Y_{ij}\otimes\frac{\p}{\p x^1}\w\ldots\widehat{\frac{\p}{\p x^i}}\ldots\widehat{\frac{\p}{\p x^j}}\ldots\w\frac{\p}{\p x^n}.
\end{aligned}
$$

\end{lemma}

\begin{proof} By lemma \ref{dimi}, there is only one generator and the rest of 
the proof is similar to the case $k=n-2$ in the proof of \cite[Lemma 4.2]{BGR} and follows the path of the proof of the previous lemma.
\end{proof}

\section{The Leibniz Homology of $\h_n$}
Recall that for the Leibniz algebra $\h,$ the Leibniz homology of $\h$ with coefficients in $\br$ denoted $HL_*(\h,\br),$ is the homology of the Loday complex $T^*(\h),$ namely $$k\stackrel{0}{\longleftarrow}\h\stackrel{[~,~]}{\longleftarrow}\h^{\otimes^2}\stackrel{d}{\longleftarrow}\ldots\stackrel{d}{\longleftarrow}\h^{\otimes^{n-1}}\stackrel{d}{\longleftarrow}\h^{\otimes^n}\leftarrow\ldots$$ where $\h^{\otimes^n}$ is the $n$th tensor power of $\h$ over $\br,$ and where
 $$
\begin{aligned} d(h_1\otimes h_2\otimes\ldots\otimes h_n)=&\\ &\sum_{1\leq i<j\leq n}(-1)^{j}h_1\otimes h_2\otimes\ldots\otimes h_{i-1}\otimes [h_i,h_j]\otimes h_{i+1}\otimes\ldots \widehat{h_j}\ldots\otimes h_n~~~~~ \cite{L}.
\end{aligned}$$ 

\begin{remark}
The antisymmetrization of the invariants  $\delta$  and $\beta$ determined in lemma \ref{delta} and lemma \ref{beta} are  cancelled in the Pirashvili spectral sequence  and  $\rho$ is not a cycle in the Leibniz complex (see the proof of \cite[theorem 5.2]{BGR}).
\end{remark}

\begin{lemma}\label{gamma}
 Let $\tilde{\gamma}_{p,q}=\bar{\gamma}_{p,q} + \bar{\gamma}'_{p,q}$ with  
 $$
\begin{aligned}\bar{\gamma}_{p,q}=&\frac{1}{n!}\big(~\sum_{\substack{1\leq i<j\leq p, \\ \sigma\in S_{n-2}}}(-1)^{i+j+1}sgn(\sigma) X_{ij}\otimes\frac{\p}{\p x^{\sigma(1)}}\otimes\ldots\widehat{\frac{\p}{\p x^{\sigma(i)}}}\ldots\widehat{\frac{\p}{\p x^{\sigma(j)}}}\ldots\otimes\frac{\p}{\p x^{\sigma(n)}}\\ &- \sum_{\substack{p+1\leq i<j\leq n\\ \sigma\in S_{n-2}}} (-1)^{i+j+p+1}sgn(\sigma) X_{ij}\otimes\frac{\p}{\p x^{\sigma(1)}}\otimes\ldots\widehat{\frac{\p}{\p x^{\sigma(i)}}}\ldots\widehat{\frac{\p}{\p x^{\sigma(j)}}}\ldots\otimes\frac{\p}{\p x^{\sigma(n)}}\\ &+ \sum_{\substack{1\leq i\leq p\\ p+1\leq j\leq n\\ \sigma\in S_{n-2}}} (-1)^{i+j}sgn(\sigma)Y_{ij}\otimes\frac{\p}{\p x^{\sigma(1)}}\otimes\ldots\widehat{\frac{\p}{\p x^{\sigma(i)}}}\ldots\widehat{\frac{\p}{\p x^{\sigma(j)}}}\ldots\otimes\frac{\p}{\p x^{\sigma(n)}}~\big)
\end{aligned}
$$   and

 $$
\begin{aligned}\bar{\gamma}'_{p,q}=&\frac{1}{n!}\big(~\sum_{\substack{1\leq i<j\leq p, \\ \sigma\in S_{n-2}}}(-1)^{i+j+1}sgn(\sigma)\frac{\p}{\p x^{\sigma(1)}}\otimes\ldots\widehat{\frac{\p}{\p x^{\sigma(i)}}}\ldots\widehat{\frac{\p}{\p x^{\sigma(j)}}}\ldots\otimes\frac{\p}{\p x^{\sigma(n)}}\otimes X_{ij}\\ &- \sum_{\substack{p+1\leq i<j\leq n\\ \sigma\in S_{n-2}}} (-1)^{i+j+p+1}sgn(\sigma) \frac{\p}{\p x^{\sigma(1)}}\otimes\ldots\widehat{\frac{\p}{\p x^{\sigma(i)}}}\ldots\widehat{\frac{\p}{\p x^{\sigma(j)}}}\ldots\otimes\frac{\p}{\p x^{\sigma(n)}}\otimes X_{ij}\\ &+ \sum_{\substack{1\leq i\leq p\\ p+1\leq j\leq n\\ \sigma\in S_{n-2}}} (-1)^{i+j}sgn(\sigma)\frac{\p}{\p x^{\sigma(1)}}\otimes\ldots\widehat{\frac{\p}{\p x^{\sigma(i)}}}\ldots\widehat{\frac{\p}{\p x^{\sigma(j)}}}\ldots\otimes\frac{\p}{\p x^{\sigma(n)}}\otimes Y_{ij}~\big).
\end{aligned}
$$  Then 
\begin{itemize}
	\item 
$\tilde{\gamma}_n$ is an $\h_n$- invariant 

\item 
$\pi_3^*([\tilde{\gamma}_n])=\pi^*_3([\bar{\gamma}_n])=[\gamma_n]$ in $H^{Lie}_{n-2}(\h_n;~\h_n)$ where $\pi_3: \h_{n-1}^{\otimes n}\longrightarrow\h_n\otimes\h^{\w n-2}$ is the projection and $\h^{\w n-2}$ is the $(n-2) ^{th}$ wedge product over $\br.$   More precisely, $\bar{\gamma}_{p,q}$ is a cycle in $\h^{\otimes (n-1)}$ homologous to the $\h_n$-invariant $\tilde{\gamma}_{p,q}.$ Thus $$K_k:=Ker[H^{Lie}_k(\I_n; ~\h_n)^{\so(p,q)}\rightarrow H^{Lie}_{k+1}(\h_n)]=\begin{cases} \left\langle \tilde{\gamma}_{p,q}\right\rangle,& \mbox{if} ~~k=n-2, \mbox{}\\0,& \mbox{else}. \mbox{}\end{cases}$$
\end{itemize}
\end {lemma}
\begin{proof}
Clearly, ${\bar{\gamma}}_{p,q}$ and $\bar{\gamma}'_{p,q}$ are  $\so(p,q)$-invariant, so is $\tilde{\gamma}_{p,q}.$ That $\tilde{\gamma}_{p,q}$ is an $\I_n$-invariant follows from the fact that  $[\frac{\p}{\p x^i}, {\bar{\gamma}}_{p,q}]=[\frac{\p}{\p x^i}, \bar{\gamma}'_{p,q}]$  for all $i=1\ldots n.$   The second assertion follows by definition of $\pi_3.$
\end{proof}

The following is the main result of the paper. It generalizes  the Leibniz homology of the Poincar\'e Lie algebra \cite[Corollary 4.5]{JL}.  

\begin{theorem}
There is an isomorphism of vector spaces $$HL_*(\h_n)\cong (\br\oplus\left\langle \tilde{\alpha}_{p,q}\right\rangle)\otimes T^*(\tilde{\gamma}_{p,q}),$$
 and an algebra isomorphism  $$HL^*(\h_n)\cong (\br\oplus\left\langle \tilde{\alpha}^d_{p,q}\right\rangle)\otimes T^*(\tilde{\gamma}^d_{p,q}),$$ 
where $$\tilde{\alpha}_n=\sum_{\sigma\in S_n}\mbox{sgn} (\sigma)\frac{\p}{\p x^{\sigma(1)}}\otimes\frac{\p}{\p x^{\sigma(2)}}\otimes\frac{\p}{\p x^{\sigma(3)}}\otimes\ldots\otimes\frac{\p}{\p x^{\sigma(n)}},$$ $$\tilde{\alpha}^d_n=\sum_{\sigma\in S_n}\mbox{sgn}(\sigma) dx^{\sigma(1)}\otimes dx^{\sigma(2)}\otimes\ldots\otimes dx^{\sigma(n)},$$   
$$
\begin{aligned}\bar{\gamma}^d_{p,q}=&\frac{1}{n!}\big(~\sum_{\substack{1\leq i<j\leq p, \\ \sigma\in S_{n-2}}}(-1)^{i+j+1}sgn(\sigma) X_{ij}\otimes dx^{\sigma(1)}\otimes\ldots\widehat{dx^{\sigma(i)}}\ldots\widehat{dx^{\sigma(j)}}\ldots\otimes dx^{\sigma(n)}\\ &- \sum_{\substack{p+1\leq i<j\leq n\\ \sigma\in S_{n-2}}} (-1)^{i+j+p+1}sgn(\sigma) X_{ij}\otimes dx^{\sigma(1)}\otimes\ldots\widehat{dx^{\sigma(i)}}\ldots\widehat{dx^{\sigma(j)}}\ldots\otimes dx^{\sigma(n)}\\ &+ \sum_{\substack{1\leq i\leq p\\ p+1\leq j\leq n\\ \sigma\in S_{n-2}}} (-1)^{i+j}sgn(\sigma)Y_{ij}\otimes dx^{\sigma(1)}\otimes\ldots\widehat{dx^{\sigma(i)}}\ldots\widehat{dx^{\sigma(j)}}\ldots\otimes dx^{\sigma(n)}~\big)
\end{aligned}
$$   and

 $$
\begin{aligned}\bar{\gamma}^d_{p,q}=&\frac{1}{n!}\big(~\sum_{\substack{1\leq i<j\leq p, \\ \sigma\in S_{n-2}}}(-1)^{i+j+1}sgn(\sigma) dx^{\sigma(1)}\otimes\ldots\widehat{dx^{\sigma(i)}}\ldots\widehat{dx^{\sigma(j)}}\ldots\otimes dx^{\sigma(n)}\otimes X_{ij}\\ &- \sum_{\substack{p+1\leq i<j\leq n\\ \sigma\in S_{n-2}}} (-1)^{i+j+p+1}sgn(\sigma) dx^{\sigma(1)}\otimes\ldots\widehat{dx^{\sigma(i)}}\ldots\widehat{dx^{\sigma(j)}}\ldots\otimes dx^{\sigma(n)}\otimes X_{ij}\\ &+ \sum_{\substack{1\leq i\leq p\\ p+1\leq j\leq n\\ \sigma\in S_{n-2}}} (-1)^{i+j}sgn(\sigma)dx^{\sigma(1)}\otimes\ldots\widehat{dx^{\sigma(i)}}\ldots\widehat{dx^{\sigma(j)}}\ldots\otimes dx^{\sigma(n)}\otimes Y_{ij}~\big)
\end{aligned}
$$
and $HL^*$ is afforded the Zinbiel algebra (dual leibniz algebra)\cite{LJ}.

\end{theorem}

\begin{proof}
We have by Lodder's structure theorem  \cite[Lemma 3.6]{JL}  that $$HL_*(\h_n)\cong [\wedge^*(\I)]^{\so(p,q)}\otimes T(K_*)~~~\mbox{where} ~~~T(K_*)=\sum_{n\geq 0}K_*^{\otimes n}$$ denotes the tensor algebra. This combined with lemma \ref{alpha} and  lemma \ref{gamma} yield  the graded vector space isomorphism $$HL_*(\h_n)\cong (\br\oplus\left\langle \tilde{\alpha}_n\right\rangle)\otimes T^*(\tilde{\gamma}_n).$$
 For the cohomology, we use the vector space isomorphism $$HL^*(\h_n;~\br)\cong Hom(HL_*(\h_n;~\br),~\br),$$ to conclude that $$HL^*(\h_n)\cong (\br\oplus\left\langle \tilde{\alpha}^d_n\right\rangle)\otimes T^*(\tilde{\gamma}^d_n),$$ where $\tilde{\alpha}^d_{p,q}$ and $\tilde{\gamma}^d_{p,q}$ are respectively  dual to $\tilde{\alpha}_{p,q}$ and $\tilde{\gamma}_{p,q}$ with respect to the basis of $\h_n.$
\end{proof}

\subsection*{Concluding Remark}
We have determined the Leibniz homology of an abelian extension $\h_n\cong\I_n\oplus\so(p,q)$ of the Lie algebra $\so(p,q)$ where  $\so(p, q)$ acts on the real vector space $\I_n$  so that $\I_n$ is the standard representation $R^n, (n = p+q).$  A similar, but more interesting result may be obtained if one considers a more general approach. More precisely, let $M$ be a Lie algebra with $[a, b] = 0$ for all $a, b \in M .$ Then the vector space $\h$ of all linear combinations of elements in $\so(p, q)$ with elements in $M$ is an abelian  extension of $\so(p,q)$ by $M$, and $ \h$ is a Lie algebra via the bracket
$$[g_1 + m_1,  g_2 + m_2] = [g_1, g_2] + [g_1, m_2] - [g_2, m_1]$$ where $g_1,g_2\in\so(p,q)$,  $m_1,m_2\in M,$ and $[g, m]$ is the action of $g$ on $M.$   This fits into the Lodder's structure theorem for Leibniz
homology and the map  $so(p, q)  \rightarrow M$ of Lie modules  \underline{\textbf{may be}}  captured as a Leibniz homology class. This was suggested by Jerry Lodder and is part of a future work.

\subsection*{Acknowledgment}
I am indebted to  Jerry Lodder for advises and I benefited from
conversations with David Vogan and  Birne Binegar about the representation theory of $\so(p,q).$


\end{document}